\documentclass[leqno,12pt]{amsart}
\usepackage{amsmath,amstext,amssymb,amsopn,amsthm,mathrsfs}

\title[Riesz transforms for Jacobi expansions]
    {Riesz transforms for Jacobi expansions}

\author[A.Nowak]{Adam Nowak}
\author[P.Sj\"ogren]{Peter Sj\"ogren}

\address{
\noindent Adam Nowak, \newline
    Institute of Mathematics and Computer Science, \newline
    Wroc\l{}aw University of Technology, \newline
    Wyb{.} Wyspia\'nskiego 27,
    PL-50--370 Wroc\l{}aw, Poland
    }
\email{anowak@pwr.wroc.pl}

\address{
\noindent Peter Sj\"ogren \newline
    Mathematical Sciences, 
    University of Gothenburg \newline
    Mathematical Sciences,
    Chalmers University of Technology \newline 
    SE-412 96 G\"oteborg, Sweden
    }
\email{peters@math.chalmers.se}

\allowdisplaybreaks

\theoremstyle{plain}
\newtheorem{thm}{Theorem}[section]

\newtheorem{lem}[thm]{Lemma}
\newtheorem{prop}[thm]{Proposition}
\newtheorem{cor}[thm]{Corollary}

\theoremstyle{definition}

\theoremstyle{remark}
\newtheorem*{rem*}{Remark}
\newtheorem{rem}[thm]{Remark}

\DeclareMathOperator{\grad}{grad}
\DeclareMathOperator{\divergence}{div}

\DeclareMathOperator{\kernel}{Ker}
\DeclareMathOperator{\domain}{Dom}

\theoremstyle{plain}

\def\R{\mathbb R}
\def\N{\mathbb N}
\def\jac{P} 		
\def\P{S}   			
\def\m{\varrho} 	
\def\ab{(\alpha,\beta)}
\def\J{\mathcal J} 
\def\JE{\mathbb J} 
\def\sqr{\Phi}

\def\e{e}
\def\d{d}
\def\Nset{\mathbb N}
\def\Rset{\mathbb R}

\begin{document}

\begin{abstract}
We define Riesz transforms and conjugate Poisson integrals
associated with multi-dimensional Jacobi expansions.
Under a slight restriction on the type parameters,
we prove that these operators are bounded in $L^p,\; 1<p<\infty,$
with constants independent of the dimension.
Our tools are suitably defined $g$-functions and Littlewood-Paley-Stein theory,
involving the Jacobi-Poisson semigroup and modifications of it.
\end{abstract}

\maketitle

\footnotetext{
\emph{\noindent 2000 Mathematics Subject Classification:} primary 42C10; secondary 42B25\\
\emph{Key words and phrases:} Jacobi expansions, Riesz transforms, g-functions,
    dimension-free estimates, conjugate Poisson integrals

Research of both authors supported by the European Commission via the Research Training Network
``Harmonic Analysis and Related Problems'', contract HPRN-CT-2001-00273-HARP.

The first-named author was also supported by MNiSW Grant N201 054 32/4285.
}


\section{Introduction}
\label{sec:introduction}

The aim of this paper is to study the Riesz transform
$\mathcal{R}^{\ab} = (R^{\ab}_1,\ldots,R^{\ab}_d)$
naturally associated with multi-dimensional Jacobi polynomial expansions of type $\ab.$
Our main result is contained in Theorem \ref{main_result_of_paper}:
we prove that if $\alpha$ and $\beta$ are multi-indices whose components $\alpha_j$ and
$\beta_j$ belong to $[-1\slash 2, \infty)$ then each
$R^{\ab}_j, \; j=1,\ldots,d,$ is bounded in $L^p$ (with the appropriate
measure) for $1 < p < \infty$, and the corresponding operator norms are
independent of the dimension $d$
and the type multi-indices $\alpha,\beta.$
As a consequence, we obtain boundedness and convergence results for the associated
conjugate Poisson integrals, see Corollary \ref{conjugate_result} below.

Our methods are analytic and based on the Littlewood-Paley-Stein theory
contained in Stein's monograph \cite{St}.
We construct appropriate square functions that relate a function and its
Riesz transform, and then prove that these square functions
satisfy two-sided $L^p$ inequalities, $1 < p < \infty.$
The same scheme was exploited by Guti\'errez \cite{Gu}, who considered Riesz transforms
associated with multi-dimensional Hermite expansions, and by one of the authors in
\cite{No}, where Riesz transforms for multi-dimensional Laguerre expansions were studied.
The case of the Jacobi expansions is certainly more complex than that of Hermite, 
but on the other hand to some extent comparable to the Laguerre case.

Riesz transforms and conjugate Poisson integrals are important
objects in harmonic analysis as well as in the theory of partial differential equations.
The study of these objects in the context of orthogonal expansions
was initiated by the fundamental work of Muckenhoupt
and Stein \cite{MuS}, which treated, among other things, one-dimensional ultraspherical
expansions. Then Muckenhoupt elaborated necessary tools and
investigated Riesz transforms (or rather conjugate mappings)
for Hermite and Laguerre
expansions \cite{Mu1,Mu2}. However, he worked in the one-dimensional
setting and used methods which are inapplicable in higher
dimensions. In fact, passing with Riesz transforms to higher
dimensions turned out not to be as straightforward as one could
expect. The first corresponding multi-dimensional result was
obtained by P.A. Meyer \cite{Me}, who proved by probabilistic
methods the $L^p$ boundedness of the Riesz-Hermite transforms in
arbitrary dimension. Later many authors gave other proofs, see the
survey \cite{Sj}. The Laguerre setting is more involved than that of
Hermite, and the $L^p$ boundedness of the multi-dimensional
Riesz-Laguerre transforms was proved recently by Guti\'errez, Incognito
and Torrea \cite{GIT}
(for half-integer type multi-indices) and by the first-named author \cite{No}
(for a continuous range of type multi-indices). Riesz transforms and
conjugate Poisson integrals associated with multi-dimensional Jacobi
polynomials, the remaining of the three classical orthogonal
polynomial systems, are treated in the present paper. 

We note that in the one-dimensional setting conjugacy for Jacobi polynomial
expansions was considered by Li \cite{Li}, and recently by Buraczewski et al.
\cite{Bur} (only the ultraspherical case). 
However, the settings considered in \cite{MuS, Li, Bur} have the common
disadvantage that the underlying differential operators have (for almost all choices of
the parameters) nontrivial zero order terms and hence the associated
semigroups are not symmetric diffusion semigroups in the sense of \cite{St}.
This obviously makes a contrast with
the settings of Hermite and Laguerre polynomial expansions.
The reason for perturbing the ``genuine'' Jacobi diffusion operator with a constant
term is purely technical and caused by the lack of an explicit expression
for either the heat or the Poisson kernel, without this modification.
Nevertheless, in this paper we overcome the difficulty and
consider the ``genuine'' Jacobi setting. Consequently, we introduce
definitions of multi-dimensional Riesz-Jacobi transforms and
corresponding conjugate Poisson integrals which in one dimension
differ somewhat from those in \cite{MuS,Li,Bur}, but on the other
hand are more natural and perfectly fit into a unified scheme of
conjugacy satisfied by all the three classical orthogonal expansions.
The crucial ingredients of this scheme (discussed in Section
\ref{sec:Riesz}) are Cauchy-Riemann type equations that link all
involved operators and systems of supplementary Riesz transforms and
conjugate Poisson integrals.
Some necessary tools here are modified versions of the Jacobi Laplacian
and the corresponding modified Jacobi-Poisson semigroups.

An interesting aspect of the $L^p$ estimates in the multi-dimensional setting is
the question whether the corresponding $L^p$ constants can be chosen independently
of the dimension, and this is related to analysis in infinite dimension.
Such dimension-free $L^p$ estimates are known to hold for Riesz-Hermite and Riesz-Laguerre
transforms. The results of our paper show that the situation is similar in the
Jacobi case
as long as each partial Riesz transform $R_j^{\ab}$ is taken separately.
When the vector $\mathcal{R}^{\ab}$ is considered, the situation gets substantially
more complicated,
and the question of dimension independence remains open.

The paper is organized as follows.
Section \ref{sec:preliminaries} contains basic facts and notation needed in the sequel.
Then Section \ref{sec:supplementary} introduces modified semigroups and crucial 
estimates between related heat kernels and the Jacobi heat kernel.
These estimates are proved by means of a parabolic PDE technique.
In Section \ref{sec:square_functions} we define suitable square functions and prove
relevant $L^p$ inequalities. The main results of this section, 
gathered in Theorem \ref{main_theorem}, are also of independent interest.
Finally, in Section \ref{sec:Riesz} the results concerning
Riesz transforms and conjugate Poisson integrals are proved. 
Also, supplementary systems of such operators are introduced and briefly studied,
complementing a conjugacy scheme for Jacobi expansions.


\section{Preliminaries}
\label{sec:preliminaries}

Given $\alpha, \beta > -1,$ the one-dimensional \emph{Jacobi polynomials}
of type $(\alpha,\beta)$ are defined by the Rodrigues formula
\[
\jac_k^{(\alpha,\beta)} (x) = \frac{(-1)^k}{2^k k!} (1-x)^{-\alpha}(1+x)^{-\beta}
    \frac{{\d}^k}{{\d}x^k} \big( (1-x)^{\alpha+k}(1+x)^{\beta+k} \big),
    \quad k \in \Nset ,\;\; -1<x<1.
\]
Note that each $\jac_k^{(\alpha,\beta)}$ is a polynomial of degree $k.$
Given multi-indices
$\alpha = (\alpha _1, \ldots , \alpha _d)$ and $\beta = (\beta_1,\ldots,\beta_d),\;$
$\alpha,\beta \in (-1,\infty)^d,$ the $d$-dimensional
Jacobi polynomials of type $(\alpha,\beta)$ are tensor products
$$
\jac_k^{(\alpha,\beta)} (x) = \prod _{i=1}^{d} \jac_{k_i}^{(\alpha_i,\beta_i)}(x_i),
\qquad k \in \Nset^d, \quad x \in (-1,1)^d.
$$
Jacobi polynomials have many interesting properties, see for instance the classical monograph 
by Szeg\"o \cite{Sz}.
In particular, cf. \cite[(4.21.7)]{Sz},
\begin{equation} \label{diff_jac_pol}
\partial_{x_i} \jac^{(\alpha,\beta)}_{k}(x) =
    \frac{1}{2} (k_i+\alpha_i+\beta_i+1)\jac^{(\alpha+e_i,\beta+e_i)}_{k-e_i}(x),
    \quad \quad i=1,\ldots,d,
\end{equation}
$e_i$ denoting the $i$th coordinate vector in $\Rset^d.$
Here and later on we use the convention that
$\jac^{(\alpha+e_i,\beta+e_i)}_{k-e_i}=0$ if $k_i-1<0.$

Consider the beta-type measure $\m_{(\alpha,\beta)}$ in $(-1,1)^d$ given by
$$
d\m_{(\alpha,\beta)}(x) = \prod^{d}_{i=1} (1-x_i)^{\alpha_i} (1+x_i)^{\beta_i} \, dx.
$$
The \emph{Jacobi differential operator}
$$
\J^{\ab} = - \sum^{d}_{i=1} \bigg[ (1-x^2_i) \partial^2_{x_i}
+ \big(\beta_i - \alpha_i - (\alpha_i + \beta_i + 2) x_i \big) \partial_{x_i} \bigg]
$$
is nonnegative and symmetric in $L^2(d\m_{\ab})$ on the domain $C^{\infty}_c((-1,1)^d).$
Each Jacobi polynomial $\jac^{\ab}_k$ is an eigenfunction of $\J^{\ab}$
with the corresponding eigenvalue
$$
\lambda_k = \lambda_k^{\ab} = \lambda_{k_1} + \ldots + \lambda_{k_d},
    \quad \textrm{where} \quad \lambda_{k_i} = k_i (k_i + \alpha_i + \beta_i + 1)
$$
(in the sequel we omit the superscript $\ab$ in $\lambda_k^{\ab}$ whenever it makes no confusion).
Moreover, the system $\{\jac^{\ab}_{k} : k \in \Nset^d\}$
constitutes an orthogonal basis in the Hilbert space $L^2(d\m_{\ab}).$
Thus, any function $f \in L^2(d\m_{\ab})$ has the expansion
\[
f = \sum_{k \in \Nset^d} a_k(f) \jac^{\ab}_k,
\]
with the Fourier-Jacobi coefficients given by
$$
a_k(f)  =  \langle f, \jac_k^{\ab}
    \rangle_{{\ab}} \slash \|\jac_k^{\ab}\|^{2}_{2,\ab};
$$
here, and also later on, we use the notation $\langle f, \jac^{\ab}_k \rangle_{\ab}
= \int f(y) \jac^{\ab}_k(y)\, d\m_{\ab}(y)$ and $\|\cdot\|_{p,\ab} \equiv \|\cdot\|_{L^p(d\m_{\ab})}$,
$1\le p \le \infty$.
The squared norm appearing above is known explicitly (cf. \cite[(4.3.3)]{Sz})
to be
\begin{equation} \label{l2norm}
\|\jac_k^{\ab}\|^{2}_{2,\ab} = \prod_{i=1}^d
    \frac{2^{\alpha_i+\beta_i+1}\Gamma(k_i+\alpha_i+1)\Gamma(k_i+\beta_i+1)}
    {(2k_i+\alpha_i+\beta_i+1)\Gamma(k_i+\alpha_i+\beta_i+1)\Gamma(k_i+1)},
\end{equation}
where for $k_i=0$ the product $(2k_i+\alpha_i+\beta_i+1) \Gamma(k_i+\alpha_i+\beta_i+1)$
must be replaced by $\Gamma(\alpha_i+\beta_i+2)$. Note that by Stirling's formula
there exists a constant $C$ such that
\begin{equation} \label{L2normest}
\|\jac_k^{\ab}\|^{-2}_{2,\ab} \le C (k_1+1) \cdot \ldots \cdot (k_d+1), \qquad k \in \N^d.
\end{equation}

The operator $\J^{\ab}$ has a self-adjoint extension (denoted by the same symbol) given by\
\begin{equation} \label{cl_J}
\J^{\ab} f = \sum_{k \in \N^d} a_k(f) \lambda_k \jac^{\ab}_k
\end{equation}
on the domain
$$
\domain \J^{\ab} = \Big\{ f \in L^2(d\m_{\ab}) : \sum_{k \in \N^d} |a_k(f)|^2
    \lambda_k^2 \|\jac^{\ab}_k\|^2_{2,\ab} < \infty \Big\}.
$$
Then the spectrum of $\J^{\ab}$ is the discrete set $\{\lambda_k : k \in \N^d\}$, and
the corresponding spectral resolution is given by \eqref{cl_J}. The inclusion
$C^{\infty}_c( (-1,1)^d ) \subset \domain \J^{\ab}$ can be easily justified since
for such functions $f$ one has $a_k(f) \lambda_k = a_k(\J^{\ab}f)$.
This identity is a consequence of the symmetry of $\J^{\ab}$ and follows by
using the divergence form of the Jacobi operator,
\begin{equation} \label{divergent}
\J^{\ab} = - \sum_{i=1}^d (1-x_i)^{-\alpha_i}(1+x_i)^{-\beta_i}
    \partial_{x_i} \big[ (1-x_i)^{\alpha_i+1} (1+x_i)^{\beta_i+1} \partial_{x_i} \big],
\end{equation}
and integrating by parts.
The same argument shows that $\domain \J^{\ab}$ contains the space $C^2_b((-1,1)^d)$
of all bounded $C^2$ functions on $(-1,1)^d$ with bounded first and second order
derivatives.

The semigroup generated by $\J^{\ab}$ is called the \emph{Jacobi semigroup}
and will be denoted by $T^{\ab}_t.$ We have for $f \in L^2(d\m_{\ab})$
\begin{equation} \label{def_heat_ser}
T^{\ab}_t f = \exp\big(-t \J^{\ab}\big)f = 
	\sum_{k \in \Nset^d} a_k(f) {\e}^{-t \lambda_k} \jac^{\ab}_k,
\end{equation}
the convergence being in $L^2(d\m_{\ab}).$
The above series may also be regarded as the definition
of $T^{\ab}_t f$ for $f \in L^1(d\m_{\ab})$
(and hence for $f \in L^p(d\m_{\ab}),$ $1 \le p \le\infty$),
since for such $f$ the Fourier--Jacobi expansion of $T^{\ab}_t f$
converges pointwise. To give a brief justification of this fact, we note that
Jacobi polynomials satisfy the estimate
(cf. \cite[(7.32.2)]{Sz})
\begin{equation} \label{jac_est}
\big| P^{\ab}_k(x) \big| \le C \,
    (k_1+1)^{\alpha_1+\beta_1+2} \cdot \ldots \cdot (k_d+1)^{\alpha_d+\beta_d+2}
\end{equation}
uniformly in $k \in \Nset^d$ and $x \in (-1,1)^d.$ Hence, with the aid of \eqref{L2normest},
it is easily seen that the absolute growth of the Fourier-Jacobi coefficients $a_k(f)$
is at most polynomial in $k$, and so the series converges absolutely due to the exponentially
decreasing factor $\e^{-t\lambda_k}.$

To obtain an integral representation of $T^{\ab}_t,$ we insert the integral defining
$a_k(f)$ into \eqref{def_heat_ser} and then, using Fubini's theorem, interchange
the order of summation and integration. The result is
\begin{equation} \label{def_heat}
T^{\ab}_t f(x) = \int G^{\ab}_t(x,y) f(y) \, d\m_{\ab}(y), \quad \quad
f \in L^1(d\m_{\ab}),
\end{equation}
where
\[
G^{\ab}_t(x,y) = \sum_{k \in \Nset^d} {\e}^{-t\lambda_k} \jac^{\ab}_k(x) \jac^{\ab}_k(y)
    \slash \|\jac_k^{\ab}\|^{2}_{2,\ab}.
\]
The above kernel is smooth for $x,y \in (-1,1)^d,\; t>0$,
and the integral in (\ref{def_heat}) is absolutely convergent.
In contrast with the Hermite and Laguerre cases, an explicit formula for
the heat kernel
$G^{\ab}_t(x,y)$ is not known.
The main obstacle in computing the kernel comes from the fact that
the Jacobi eigenvalues $\lambda_k$ are not linearly distributed.
Nevertheless, $G^{\ab}_t(x,y)$ was proved to be strictly positive for
$x,y \in (-1,1)^d,\; t>0,$ by Karlin and McGregor \cite{KM}.
The positivity also follows from more general results by Beurling and Deny,
see \cite[Section 1.3]{Da}.

It is well known that $T_t^{\ab}$ is a \emph{symmetric diffusion semigroup}
in the sense of \cite[Chapter 3]{St}
(in fact $T_t^{\ab}$ is a transition semigroup for the Jacobi diffusion
process, which already received attention, cf. \cite{KM} and references there).
In particular, $T_t^{\ab} {\bf{1}} = {\bf{1}}$ and
\[
\|T^{\ab}_t f\|_{p,\ab} \le \|f\|_{p,\ab},
\quad \quad 1 \le p \le \infty.
\]

The corresponding Poisson semigroup $\P^{\ab}_t = \exp(-t (\J^{\ab})^{1\slash 2})$ is for $f \in L^1(d\m_{\ab})$ defined by
\[
\P^{\ab}_t f = \sum_{k \in \Nset^d} a_k(f) {\e}^{-t \lambda_k^{1\slash 2}} \jac^{\ab}_k.
\]
Like \eqref{def_heat_ser}, the series is pointwise absolutely convergent.
Now, using the identity
$$
e^{-t\gamma} = \frac{1}{\sqrt{\pi}} \int_0^{\infty} \frac{e^{-u}}{\sqrt{u}}
    e^{-\gamma^2 t^2\slash(4u)} \, du, \qquad t>0, \quad \gamma \ge 0,
$$
and Fubini's theorem, we express $\P^{\ab}_t$
as a weighted average of $T^{\ab}_t:$
\[
\P^{\ab}_t f(x) = \frac{1}{\sqrt{\pi}} \int ^{\infty}_0 \frac{{\e}^{-u}}{\sqrt{u}}
T^{\ab}_{t^2 \slash (4u)} f(x) \,du, \quad \quad f \in L^1(d\m_{\ab}).
\]
This is usually referred to as the subordination formula or principle.

By general theory (see \cite[p.\,73]{St}) it follows that for $1<p \le \infty$ the maximal operators
$T^{\ab}_{*} f(x) = \sup_{t>0} \big| T^{\ab}_t f(x) \big|$ and
$\P^{\ab}_{*} f(x) = \sup_{t>0} \big| \P^{\ab}_t f(x) \big|$ satisfy
\begin{equation} \label{max_ineq}
\big{\|} T^{\ab}_{*} f \big{\|}_{p,\ab} +
\big{\|} \P^{\ab}_{*} f \big{\|}_{p,\ab} \le C_p \; \|f\|_{p,\ab},
\quad \quad f \in L^p(d\m_{\ab}).
\end{equation}
Let us emphasize that the constant $C_p$ depends neither on the dimension $d$ nor on
the type multi-indices $\alpha,\beta.$ An important consequence of (\ref{max_ineq})
and the fact that the Jacobi polynomials span $L^p(d\m_{\ab}),\; 1 < p < \infty,$
is
\begin{equation} \label{max_conv}
\lim_{t\to 0^{+}} T^{\ab}_t f(x) = \lim_{t\to 0^{+}} \P^{\ab}_t f(x) = f(x) \;\; \textrm{a.e.},
    \quad \quad f \in L^p(d\m_{\ab}), \quad 1 < p < \infty.
\end{equation}
Similarly, since continuous functions on $[-1,1]^d$ may be uniformly approximated by polynomials,
hence by linear combinations of Jacobi polynomials, it follows by \eqref{max_ineq}
specified to $p=\infty$ that for $f \in C([-1,1]^d)$ the convergence in \eqref{max_conv} is uniform
with respect to  $x \in (-1,1)^d$.
These facts will be used later without further mention.

We define the $i$th partial derivative associated with $\J^{\ab}$ by
\[
\delta_i = \sqr_i \partial_{x_i},
\]
with the coefficient function given on $(-1,1)^d$ by $\Phi_i(x)=\Phi(x_i)$, where
$$
\sqr(t) = \sqrt{1-t^2}.
$$
A reason for using such derivatives is the following
(for further motivation see Lemma \ref{lem_1} and \eqref{mult_form}).
The formal adjoint of $\delta_i$ in $L^2(d\m_{\ab})$ is given by
\[
\delta^{*}_i = - \sqr_i \partial_{x_i} + (\alpha_i+1\slash 2)
    \sqrt{\frac{1+x_i}{1-x_i}} - (\beta_i+1\slash 2) \sqrt{\frac{1-x_i}{1+x_i}},
\]
and we have the factorization
\[
\J^{\ab} = \sum _{i=1}^d \delta^{*}_i \delta_{i}.
\]
The last identity may be written in a compact form
\[
\J^{\ab} = \divergence_{\ab} \grad_{\ab},
\]
where $\grad_{\ab} = (\delta_1,\ldots, \delta_d)$ and
$\divergence_{\ab} F = \sum^d_{i=1} \delta^{*}_i f_i$
for a vector-valued function $F(y)= (f_1(y),\ldots ,f_d(y)).$

The Riesz-Jacobi transform $\mathcal{R}^{\ab} = (R^{\ab}_1, \ldots, R^{\ab}_d)$ is
then formally defined by
\begin{equation} \label{def_riesz}
\mathcal{R}^{\ab} = \grad_{\ab} \big( \J^{\ab} \big)^{-1\slash 2} \Pi_0,
\end{equation}
where $\Pi_0$ denotes the orthogonal projection onto $(\kernel \J^{\ab})^{\perp},$
the orthogonal complement of the subspace of $L^2(d\varrho_{\ab})$ consisting of all 
constant functions.
Note that (\ref{def_riesz}) makes sense for Jacobi polynomials (hence for all polynomials)
and that by (\ref{diff_jac_pol}) we have
\begin{equation} \label{riesz_pol}
R_i^{\ab} \jac_k^{\ab} = \frac{1}{2}\lambda_k^{-1\slash 2} (k_i + \alpha_i+\beta_i+1)
    \sqr_i \jac^{(\alpha+e_i,\beta+e_i)}_{k-e_i}, \quad \quad k_i>0,
\end{equation}
and $R_i^{\ab} \jac_k^{\ab} = 0$ if $k_i=0.$
A crucial observation which must be made here is that $R_i^{\ab} \jac_k^{\ab}$ is not
a polynomial, which is a consequence of the action of the Jacobi derivatives
$\delta_i$ on $\jac_k^{\ab}.$
This effect (which is absent in the Hermite, but present in the Laguerre setting) makes the
analysis more complex, involving $d$ auxiliary orthogonal systems and semigroups.


\section{The supplementary semigroups}
\label{sec:supplementary}

We introduce additional semigroups $\widetilde{\P}^{\ab,i}_t, \; i=1,\ldots,d,$
generated by slight modifications of the operator
$(\J^{\ab})^{1\slash 2}.$ As we shall see, they play an essential role in
the study of Riesz transforms and conjugacy for Jacobi expansions.
The modified Poisson semigroups are needed since
the Jacobi derivatives $\delta_i$ do not commute with the Jacobi-Poisson semigroup.
Indeed, they make an essential step possible, namely swapping
the order of the operators in $\delta_i P^{\ab}_t,$ see \eqref{ex_ord} below.

To proceed, we first define the modified Jacobi operators
\[
M^{\ab}_i = \J^{\ab} + [\delta_i,\delta^{*}_i], \quad \quad i=1,\ldots,d,
\]
where the commutators $[\delta_i,\delta^{*}_i]= \delta_i \delta^{*}_i - \delta^{*}_i \delta_i$
are easily computed to be
$$
[\delta_i,\delta^{*}_i] = \frac{\alpha_i + 1\slash 2}{1-x_i} + \frac{\beta_i + 1\slash 2}{1+x_i}.
$$
Observe that each $M^{\ab}_{i}$ is symmetric and nonnegative in $L^2(d\m_{\ab})$ on the domain
$C^{\infty}_c( (-1,1)^d )$, since for such functions $f$
\[
\langle M^{\ab}_i f,f \rangle _{{\ab}} =
\int \Big( |\delta^{*}_i f|^{2} + \sum_{j \neq i} |\delta_j f|^{2} \Big) d\m_{\ab}.
\]

The following simple lemma is crucial.
\begin{lem} \label{lem_intr}
Given $i=1,\ldots,d,$ the functions $\sqr_i \jac^{(\alpha+e_i,\beta+e_i)}_{k}$
are eigenfunctions of $M^{\ab}_{i},$ with eigenvalues $\lambda^{\ab}_{k+e_i}.$
Moreover, the system
$$
\big\{\sqr_i \jac^{(\alpha+e_i,\beta+e_i)}_{k} : k\in \Nset^d\big\}
$$
forms an orthogonal basis in $L^2(d\m_{\ab}).$
\end{lem}

\begin{proof}
The first part follows by a direct computation, using the decomposition
$$
M^{\ab}_i = \delta_i \delta_i^{*} + \sum_{j \neq i} \delta_j^{*} \delta_j
$$
and \eqref{diff_jac_pol}, rewritten as
\begin{equation} \label{jac_diff}
\sqr_i \jac^{(\alpha+e_i,\beta+e_i)}_{k} =
        2 (k_i + \alpha_i + \beta_i + 2)^{-1} \delta_i \jac^{\ab}_{k+e_i}.
\end{equation}
Indeed, if $j \neq i$ then $\delta_j^{*} \delta_j$ is the one-dimensional Jacobi operator
in the $j$th coordinate, and hence
$$
\delta_j^{*}\delta_j \big(\sqr_i \jac^{(\alpha+e_i,\beta+e_i)}_{k} \big) =
\lambda_{k_j} \sqr_i \jac^{(\alpha+e_i,\beta+e_i)}_{k}.
$$
To handle $\delta_i \delta_i^{*}$ we write
\begin{align*}
\delta_i \delta_i^{*} \big(\sqr_i \jac^{(\alpha+e_i,\beta+e_i)}_{k} \big) & =
    2 (k_i + \alpha_i + \beta_i + 2)^{-1} \delta_i \delta_i^{*} \delta_i \jac^{\ab}_{k+e_i} \\
& = 2 (k_i + \alpha_i + \beta_i + 2)^{-1} \lambda_{k_i+1} \delta_i \jac^{\ab}_{k+e_i} \\
& = \lambda_{k_i+1} \sqr_i \jac^{(\alpha+e_i,\beta+e_i)}_{k}.
\end{align*}
The second part is a consequence of the fact that the system
$\{\jac^{(\alpha+e_i,\beta+e_i)}_{k} : k\in \Nset^d\}$
is an orthogonal basis in $L^2(d\m_{(\alpha+e_i,\beta+e_i)}).$
\end{proof}

Therefore, given $i \in \{1,\ldots,d\},$ any $f \in L^2(d\m_{\ab})$ has the expansion
\[
f = \sum_{k \in \Nset^d} a_k^i(f) \sqr_i \jac^{(\alpha+e_i,\beta+e_i)}_{k},
\]
with $a^i_k(f) = \big\langle f, \sqr_i \jac_k^{(\alpha+e_i,\beta+e_i)}
    \big\rangle_{{\ab}} \slash
     \|\sqr_i\jac^{(\alpha+e_i,\beta+e_i)}_k\|^2_{2,\ab}.$

Each of the operators $M^{\ab}_i, \; i = 1, \ldots, d,$ has a self-adjoint extension (which we
still denote by the same symbol) given by
\begin{equation} \label{M_ext}
M^{\ab}_i f = \sum_{k \in \N^d} a^i_k(f) \lambda_{k + e_i} \sqr_i
     \jac^{(\alpha+e_i,\beta+e_i)}_k
\end{equation}
on the domain
$$
\domain M^{\ab}_i = \Big\{ f \in L^2(d\m_{\ab}) : \sum_{k \in \N^d} |a_k^i(f)|^2 \lambda_{k+e_i}^2
    \big\|\sqr_i\jac^{(\alpha+e_i,\beta+e_i)}_k\big\|^2_{2,\ab}
     < \infty \Big\};
$$
the inclusions $C^{\infty}_c((-1,1)^d) \subset \domain M^{\ab}_i$ are justified like
the analogous relation for $\J^{\ab}$. Then the spectrum of
$M^{\ab}_i$ is the discrete set $\{ \lambda_{k+e_i} : k \in \N^d \}$, and the spectral
decomposition of $M^{\ab}_i$ is given by \eqref{M_ext}.

\begin{rem}
It is perhaps worth noticing that when $\alpha_i = \beta_i = -1\slash 2$ for some
$i=1,\ldots,d,$ the operators $M^{\ab}_i$ and $\J^{\ab}$ coincide as differential
operators. However, the self-adjoint extensions described above are different, since
the corresponding spectra are not equal. 
The situation is best understood in one dimension by means of the
change of variable $x = \cos \theta$. Then the Jacobi measure 
$d\m_{\ab}$ becomes Lebesgue measure $d\theta$ in $(0,\pi)$,
and the differential operator will be simply $-d^2\slash d\theta^2$.
From \cite[(4.1.7)]{Sz} we have for $k = 0,1,...$
\[
P_k^{(-1/2,-1/2)}(x) = c_k \cos k\theta
\]
and
\[
P_k^{(1/2,1/2)}(x) = c'_k \,\frac{\sin (k+1)\theta}{\sin\theta},
\]
where $x=\cos\theta$ and $c_k$ and $c'_k$ are constants. This means that
the relevant eigenfunction expansion, expressed in the
$\theta$ variable, is simply the Fourier cosine or sine series expansion,
respectively, in $(0,\pi)$. The self-adjoint extensions $\J^{\ab}$
and ${M}_1^{\ab}$ are then defined by differentiating termwise
twice the cosine or sine series, respectively. An $L^2$ function is in
the domain of the extension precisely when the corresponding differentiated
series defines an $L^2$ function. These two domains do not coincide.
Indeed, the constant function $\boldsymbol{1}$ has cosine series $1$ and sine series
\[
\sum_{k=1}^\infty \frac4{\pi(2k-1)}\sin(2k-1)\theta.
\]
Differentiating, we see that $\boldsymbol{1}$ is in the domain of  $\J^{\ab}$
but not in that of ${M}_1^{\ab}$.
\end{rem}

We set
\begin{equation*}
\widetilde{T}^{\ab,i}_t  = \exp\big({-t M^{\ab}_i}\big), \qquad
\widetilde{\P}^{\ab,i}_t = \exp\big({-t(M^{\ab}_i)^{1 \slash 2}}\big), \quad \quad i=1,\ldots,d.
\end{equation*}
Clearly, for $f \in L^2(d\m_{\ab})$ we have
\begin{align} \label{series_def_mod_T}
\widetilde{T}^{\ab,i}_t f & = \sum_{k \in \Nset^d} a_k^i(f)
    {\e}^{-t\lambda_{k+e_i}} \sqr_i \jac^{(\alpha+e_i,\beta+e_i)}_{k}, \\
\label{def_mod_poisson}
\widetilde{\P}_t^{\ab,i} f & = \sum_{k \in \Nset^d} a_k^i(f)
    {\e}^{-t\lambda_{k+e_i}^{1\slash 2}} \sqr_i \jac^{(\alpha+e_i,\beta+e_i)}_{k}.
\end{align}
The above series are also appropriate for defining the operators in question on $L^1(d\m_{\ab})$,
as in the case of $T^{\ab}_t.$ We have the integral representations
\begin{equation} \label{mod_sem_heat_def}
\widetilde{T}_t^{\ab,i}f(x) = \int \widetilde{G}_t^{\ab,i}(x,y) f(y) \,
d\m_{\ab}(y), \quad \quad i =1,\ldots,d,
\end{equation}
where
\begin{align*}
\widetilde{G}_t^{\ab,i}(x,y) 
&    = \sum_{k \in \Nset^d} \frac{{\e}^{-t \lambda^{\ab}_{k+e_i}} \sqr_i(x)
        \jac^{(\alpha+e_i,\beta+e_i)}_k(x) \sqr_i(y)
        \jac^{(\alpha+e_i,\beta+e_i)}_k(y)}
        {\|\sqr_i\jac^{(\alpha+e_i,\beta+e_i)}_k\|^2_{2,\ab}}\\
&        =
    {\e}^{-t(\alpha_i+\beta_i+2)} \sqr_i(x)\sqr_i(y) \sum_{k \in \Nset^d}\frac{
        {\e}^{-t\lambda_k^{(\alpha+e_i,\beta+e_i)}} \jac^{(\alpha+e_i,\beta+e_i)}_k(x)
        \jac^{(\alpha+e_i,\beta+e_i)}_k(y)}
        {\|\jac^{(\alpha+e_i,\beta+e_i)}_k\|^2_{2,(\alpha+e_i,\beta+e_i)}}\\
&        =
        {\e}^{-t(\alpha_i+\beta_i+2)} \sqr_i(x)\sqr_i(y) G_t^{(\alpha+e_i,\beta+e_i)}(x,y).
\end{align*}
As in (\ref{def_heat}), the integral in (\ref{mod_sem_heat_def}) converges absolutely
for $f \in L^1(d\m_{\ab}).$
A connection between $\widetilde{T}^{\ab,i}_t$ and $\widetilde{\P}^{\ab,i}_t$ is given by
the subordination formula
\begin{equation} \label{sub_tild}
\widetilde{\P}^{\ab,i}_t f(x) = \frac{1}{\sqrt{\pi}} \int ^{\infty}_0 \frac{{\e}^{-u}}{\sqrt{u}}
\widetilde{T}^{\ab,i}_{t^2 \slash (4u)} f(x) \,du, \quad \quad i=1,\ldots,d.
\end{equation}

\begin{prop} \label{diff_equations}
Let $f \in L^1(d\m_{\ab})$ and $i \in \{1,\ldots,d\}.$
Then $T^{\ab}_t f(x),$  $\P^{\ab}_t f(x),$  $\widetilde{T}^{\ab,i}_t f(x)$ and
$\widetilde{\P}^{\ab,i}_t f(x)$ are $C^{\infty}$ functions of $(t,x) \in (0,\infty) \times (-1,1)^d.$
Moreover,
\[
\big(\partial_t + \J^{\ab} \big) T^{\ab}_t f(x)
\; = \; 0 \; = \;
\big(\partial_t + M^{\ab}_i \big) \widetilde{T}^{\ab,i}_t {f}(x)
, \quad \quad t>0, \quad x \in (-1,1)^d,
\]
\[
\big(\partial^2_t - \J^{\ab} \big) \P^{\ab}_t f(x)
\; = \; 0 \; = \;
\big(\partial^2_t - M^{\ab}_i \big) \widetilde{\P}^{\ab,i}_t {f}(x),
\quad \quad t >0, \quad x \in (-1,1)^d.
\]
\end{prop}

\begin{proof}
We consider only $\widetilde{T}^{\ab,i}_t f(x)$, given by the series in (\ref{series_def_mod_T}),
since the treatment of the remaining functions is similar. 
Observe that by \eqref{jac_est} and \eqref{L2normest} the coefficients
$|a^i_k(f)|$ grow at most polynomially in $k.$ Furthermore, in view of \eqref{jac_est},
the quantity 
$$
\sup_{x \in (-1,1)^d} \big|\sqr_i(x)\jac_k^{(\alpha+e_i,\beta+e_i)}(x)\big|
$$
also has polynomial growth in $k$. Thus the series \eqref{series_def_mod_T}
may be differentiated term by term with respect to $t$, repeatedly. The result is
\begin{equation} \label{diff_res}
  \partial_t^m \widetilde{T}^{\ab,i}_t f(x) =
    \sum_{k \in \Nset^d} a_k^i(f) (-1)^m \lambda_{k+e_i}^m
    {\e}^{-t\lambda_{k+e_i}} \sqr_i(x) \jac^{(\alpha+e_i,\beta+e_i)}_{k}(x),
\end{equation}
and the right-hand side is continuous since the series converges uniformly in $(t,x)$
on compact subsets of $(0,\infty)\times (-1,1)^d.$
Using (\ref{diff_jac_pol}) we see that, for a fixed compact set $K \subset (-1,1)^d$, also
$$
\sup_{x \in K} \big|\partial_{x_j} \big(\sqr_i(x)\jac_k^{(\alpha+e_i,\beta+e_i)}(x)\big)\big|
$$
grows in $k$ not faster than polynomially.
Hence, we may differentiate the series in (\ref{diff_res}) with respect to $x_j$ term by term,
the result being a continuous function since the convergence is again locally uniform.
The same arguments apply to higher derivatives, so
$\widetilde{T}^{\ab,i}_t f(x)$ is smooth on $(0,\infty) \times (-1,1)^d.$ The corresponding heat
equation is easily verified by means of the differentiated series.
\end{proof}

\begin{lem} \label{contr}
Fix $i \in \{1,\ldots,d\}.$ If $\alpha_i \ge -1\slash 2$ and $\beta_i \ge -1\slash 2$ then
\begin{equation} \label{G_ineq}
\widetilde{G}_t^{\ab,i}(x,y) \le G_t^{\ab}(x,y), \qquad x,y \in (-1,1)^d, \quad t>0.
\end{equation}
\end{lem}

\begin{proof}
Recall that we have $\widetilde{G}_t^{\ab,i}(x,y) =
{\e}^{-t(\alpha_i+\beta_i+2)} \sqr_i(x)\sqr_i(y) G_t^{(\alpha+e_i,\beta+e_i)}(x,y).$
Observe that due to the product structure of the kernels involved, it suffices to prove
the lemma in the one-dimensional case.
Then it is enough to show that for any nonnegative $f \in C^{\infty}_c((-1,1))$
which is not identically equal to $0$ one has
$$
\sqr(x) T_t^{(\alpha+1,\beta+1)}( f\slash \sqr)(x) \le
    {\e}^{t(\alpha+\beta+2)} T_t^{\ab}f(x), \qquad t>0, \quad x \in (-1,1).
$$
Denote by $u=u(t,x)$ the left-hand side of the above inequality
and let
$$
v=v(t,x)=e^{t\eta} e^{t(\alpha+\beta+2)}T^{\ab}_t(f+\eta)(x)
$$
for some fixed $\eta >0.$
Since $f$ is smooth, both the functions $u$ and $v$ have continuous
extensions to $[0,\infty) \times (-1,1).$ Our task will be done once we show that
\begin{equation} \label{rel_uv}
u(t,x) \le v(t,x), \qquad x \in (-1,1)
\end{equation}
for all $t \ge 0.$ Let
$$
T = \sup \big\{ t' \ge 0 : u(t,x) \le v(t,x) \; \textrm{for} \; (t,x) \in
[0,t') \times (-1,1) \big\}.
$$
Clearly $u(0,x) < v(0,x)$ for $x \in (-1,1)$.
Moreover, $u(t,x) < v(t,x)$ for all $t\ge 0$ provided that $|x|$ is
sufficiently close to $1$; this is because $u(t,x) < C\Phi(x)$ and $v(t,x)>\eta$ for
$t \ge 0$, $x\in (-1,1)$.
Hence for $t$ small enough $u(t,x) < v(t,x), \; x \in (-1,1),$ which means that $T>0.$

Suppose that $T$ is finite. We shall then derive a contradiction which will end the reasoning.
Observe that $u(T,x) \le v(T,x)$ for all $x \in (-1,1)$
and $u(T,x_0) = v(T,x_0)$ for some $x_0$.
We claim that
\begin{equation} \label{d_in}
\partial_t \big(v(t,x)-u(t,x)\big)\big|_{(t,x)=(T,x_0)} > 0.
\end{equation}
This would imply that $v(t,x_0)-u(t,x_0) < 0$ for $t$ slightly less than $T$,
a contradiction.

To prove the claim, we compute the derivative in \eqref{d_in}. With the aid of the heat
equation (see Proposition \ref{diff_equations}), we get
$$
\partial_t \big(v(t,x)-u(t,x)\big) = (\alpha + \beta +2 + \eta)
    v(t,x) - \J^{\ab} v(t,x) +
    \sqr(x) \J^{(\alpha+1,\beta+1)} \big( {u(t,x)}\slash {\sqr(x)} \big).
$$
Then using the definition of $\J^{\ab}$ and the fact that $v-u = \partial_x(v-u)=0$ at 
the point $(T,x_0)$, we find that the left-hand side in \eqref{d_in} is equal to
$$
\sqr(x_0)^2\partial_x^2(v-u)(T,x_0) +
    \sqr(x_0)^{-2} \big[ \alpha + \beta + 1 + (\alpha-\beta) x_0 \big] u(T,x_0)
    + \eta u(T,x_0).
$$
The first term above is nonnegative, since the function
$x \mapsto v(T,x)-u(T,x)$ has a local minimum at $x=x_0$.
The factor in square brackets is obviously not smaller than $\alpha+\beta+1-|\alpha-\beta|$,
an expression which equals either $2\alpha+1$ or $2\beta+1$ and is nonnegative by the assumption
$\alpha,\beta \ge -1\slash 2$.
Finally, $u(T,x_0)$ is strictly positive by the corresponding property of
the kernel involved. The claim follows.
\end{proof}
An immediate consequence of Lemma \ref{contr} and \eqref{sub_tild} is the following
\begin{cor} \label{cor_1}
Let $i \in \{1,\ldots,d\}$ and assume that $\alpha_i,\beta_i \in [-1\slash 2, \infty).$ 
Then for each function $f \ge 0,$ we have
\[
\widetilde{\P}_t^{\ab,i}f(x) \le \P^{\ab}_t f(x), \quad\quad x \in (-1,1)^d,
\quad t >0,
\]
and $\{\widetilde{\P}^{\ab,i}_t\}$ is a semigroup of
contractions in $L^p(d\m_{\ab}), \; 1 \le p \le \infty.$
\end{cor}

\begin{rem} \label{rem_low}
When $\alpha_i < -1\slash 2$ or $\beta_i < -1\slash 2$, the inequality of Lemma
\ref{contr} does not hold. This is justified as follows. As in the proof of
Lemma \ref{contr}, it suffices to consider the one-dimensional case.
Take a function $f \in C_c^{\infty}((-1,1))$ such that
$0 \le f(x) \le 1$ for $x \in (-1,1)$ and $f(x)=1$ for $x \in [-1+\varepsilon,1-\varepsilon]$;
here $0 < \varepsilon <1$ will be fixed in a moment. Consider the functions
\begin{align*}
u(t,x) & = \sqr(x)\, T_t^{(\alpha+1,\beta+1)} ( {f}\slash{\sqr})(x), \\
v(t,x) & = \e^{t(\alpha + \beta + 2)} T^{\ab}_t \boldsymbol{1}(x) = e^{t(\alpha + \beta + 2)}.
\end{align*}
Clearly, both $u$ and $v$ are continuous on $[0,\infty)\times (-1,1)$ and, moreover,
$u(0,x) = v(0,x) = 1$ for $x \in [-1+\varepsilon,1-\varepsilon]$.
The derivative $\partial_t v$ is again a continuous function of $(t,x) \in [0,\infty)
\times (-1,1)$, and the same is true for $\partial_t u$ since
$$
\partial_t u(t,x) = -\sqr(x)\, T_t^{(\alpha+1,\beta+1)}
    \big( \J^{(\alpha+1,\beta+1)} (f \slash \sqr) \big)(x),
$$
by the heat equation and the fact that $T_t^{(\alpha+1,\beta+1)}$ commutes with
$\J^{(\alpha+1,\beta+1)}$ on its domain.
Thus we have
$$
\partial_t (v-u)(0,x) = \alpha + \beta + 2 + \sqr(x)\J^{(\alpha+1,\beta+1)}
\big( f\slash \sqr \big)(x).
$$
One computes that for $x \in [-1+\varepsilon,1-\varepsilon]$ the last expression is
equal to
$$
\frac{\alpha+1\slash 2}{1-x} + \frac{\beta + 1\slash 2}{1+x}.
$$
Now we fix $\varepsilon > 0$ such that this function of $x$ is strictly negative on a closed interval
$\mathcal{D} \subset [-1+\varepsilon,1-\varepsilon]$ of nonzero length (this is possible since
$\min\{\alpha,\beta\} < -1\slash 2$). It follows that $u(t,x) > v(t,x)$ for $x \in \mathcal{D}$
and small $t>0.$ Hence \eqref{G_ineq} cannot hold, which ends Remark \ref{rem_low}.
\end{rem}

An important conclusion of the above reasoning is that if $\alpha_i < -1\slash 2$ or
$\beta_i < -1/2$ for some $i$, then $\widetilde{T}^{\ab,i}_t$ are not contractions
on $L^{\infty}((-1,1)^d)$ for small $t>0.$ Since $\|\cdot\|_{p,\ab} \to \|\cdot\|_{\infty}$
as $p \to \infty$, we see that $\widetilde{T}^{\ab,i}_t$ are not contractions on $L^p(d\m_{\ab})$
for $t$ sufficiently small and $p$ large enough, whenever $\min \{\alpha_i,\beta_i\} < -1\slash 2$.
A similar behavior should be expected for
$\{\widetilde{\P}^{\ab,i}_t\}$, but this seems to require a distinct detailed analysis.

In what follows, we use the convention that constants may change their value
(but not the dependence)
from one occurrence to the next. The notation $c_p$ means that the constant depends
\emph{only} on $p$ (in particular, $c_p$ is independent of the dimension $d$ and
the type multi-indices $\alpha,\beta$). Constants are always strictly positive and finite.


\section{Square functions}
\label{sec:square_functions}

Define the joint Jacobi gradient
$$
\nabla_{\ab} = \big(\partial_t, \grad_{\ab}\big).
$$
We consider the following Littlewood-Paley-Stein type square functions:
\begin{align*}
g(f)(x) & = \bigg(\int^{\infty}_0 t \big|\nabla_{\ab} \P^{\ab}_t f(x)\big|^2 \, dt
 \bigg)^{1\slash 2}, \\
\widetilde{g}_i(f)(x) & = \bigg( \int_0^{\infty} t
\big|\partial_t \widetilde{\P}^{\ab,i}_t f(x)\big|^2 \, dt \bigg)^{1\slash 2}, 
\quad \quad i =1,\ldots , d.
\end{align*}
The main result of this section reads as follows.
\begin{thm} \label{main_theorem}
Let $1<p<\infty$ and $\alpha, \beta \in [-1\slash 2,\infty)^d.$
Then
\begin{itemize}
\item[(a)] for all $f \in L^p(d\m_{\ab})$,
\begin{equation*}
\|g(f)\|_{p,\ab} \le c_p  \|f\|_{p,\ab};
\end{equation*}
\item[(b)] given $i \in \{1,\ldots,d\}$,
\[
c_p^{-1} \|f\|_{p,\ab} \le
\|\widetilde{g}_i (f)\|_{p,\ab} \le c_p \|f\|_{p,\ab}
\]
for all $f \in L^p(d\m_{\ab}).$
\end{itemize}
\end{thm}

The case when some $\alpha_{i}$ or $\beta_{i}$ is not in $[-1/2, \infty)$ is not
covered by our results and seems to require a more subtle treatment.
The reason for this is that the inequality between the kernels \eqref{G_ineq} holds only
when $\alpha_i,\beta_i \ge -1\slash 2$, see Remark \ref{rem_low}.
Without this relation it is harder to compare
$\widetilde{S}_{t}^{(\alpha,\beta),i}$ with ${S_{t}^{(\alpha,\beta)}}$
for nonnegative $f$, which is an essential step in our entire argument.
Notice that the critical point $-1/2$ appears also when other aspects of Jacobi
expansions and Jacobi polynomials are studied, see for instance Askey's monograph \cite{As}.

\begin{rem}
Assume that $\alpha,\beta \in (-1,\infty)^d$, $1<p<\infty$, and let
$$
g_{V}(f)(x) = \bigg( \int_0^{\infty} t \big| \partial_t \P_t^{\ab} f(x) \big|^2 \,dt
    \bigg)^{1\slash 2}
$$
be the ``vertical'' $g$-function associated with the Jacobi-Poisson semigroup.
It follows by the general Littlewood-Paley theory for semigroups
(cf. \cite[Chapter 4, Sections 5 and 6]{St}) that the two-sided inequality
$$
c_p^{-1} \|f\|_{p,\ab} \le \|g_{V}(f)\|_{p,\ab}
    + \| \Pi_0^{\perp} f\|_{p,\ab} \le c_p \|f\|_{p,\ab}
$$
holds for all $f \in {L^p(d\m_{\ab})}$; here
$\Pi_0^{\perp}f = \int f d\m_{\ab} \slash \int d\m_{\ab}$ coincides on $L^2(d\m_{\ab})$
with the orthogonal projection onto the space spanned by constant functions.
Since obviously $g_{V}(f) \le g(f)$, this shows that the lower bound in 
Theorem \ref{main_theorem} (a),
$$
c_p^{-1} \|f\|_{p,\ab} \le \|g(f)\|_{p,\ab}, \qquad
f \in {L^p(d\m_{\ab})},
$$
holds provided that $\int f d\m_{\ab} = 0$.
\end{rem}

\begin{proof}[Proof of Theorem \ref{main_theorem} (b)]
In view of the results in Section \ref{sec:supplementary}, 
this two-sided, dimension-free inequality
is a direct consequence of existing results. More precisely,
since for $\alpha,\beta \in [-1\slash 2, \infty)^d$ the semigroups
$\widetilde{\P}^{\ab,i}_t,\; i=1,\ldots,d,$
form positive symmetric contraction semigroups (see Corollary \ref{cor_1}), these
inequalities follow from the refinement of Stein's general Littlewood-Paley
theory \cite{St} due to Coifman, Rochberg and Weiss \cite{CRW}; 
see also Meda \cite[Theorem 2]{M}. 
\end{proof}

The remaining part of this section is devoted to the proof of Theorem \ref{main_theorem} (a).
For a $C^2$ function $F=F(t,x)$ define
\[
\JE^{\ab} F(t,x) = \partial_t^2 F(t,x)
        - \J^{\ab} F(t,x).
\]
We will need several technical lemmas.
\begin{lem} \label{lem_1}
Let $F=F(t,x)$ be a $C^2$ function mapping $(0,\infty) \times (-1,1)^d$ into $(0,\infty)$ such
that $\JE^{\ab} F = 0.$ Then for any $p \ge 1$ we have
\[
\JE^{\ab} (F^p) = p(p-1)F^{p-2} |\nabla _{\ab}F|^2.
\]
\end{lem}

\begin{proof}
The result follows by an elementary computation.
\end{proof}

\begin{lem} \label{lem_parts}
Let $F \colon (0,\infty) \times (-1,1)^d \mapsto (0,\infty)$ be a $C^2$ function
such that $\JE^{\ab} F \ge 0$ or
$\int_0^{\infty} \int t \big|\JE^{\ab} F(t,x)\big|
\, d\m_{\ab}(x)dt < \infty.$
Assume that
\begin{enumerate}
\item[(a)]
$\sup \{ |F(t,x)| : t>0, x\in (-1,1)^d \}  < \infty;$
\item[(b)]
$\sup \{ |\nabla_{\! x} F(t,x)| :   t>0, x\in (-1,1)^d \}  < \infty;$
\item[(c)]
$t | \partial_t F(t,x)| \le  \phi(t)$ for all $t>0,$
where the function $\phi$ is continuous, vanishes at $0$ and $\infty,$
and satisfies $\int_0^{\infty} t^{-1} \phi(t) \, dt < \infty.$
\end{enumerate}
Then for each $x$ the limits
$F(0,x) = \lim_{t \to 0^{+}} F(t,x)$ and $F(\infty,x) = \lim_{t \to \infty} F(t,x)$ exist, and
\begin{equation*}
\int_0^{\infty} \int t \; \JE^{\ab} F(t,x) \, d\m_{\ab}(x)dt 
 =
\int F(0,x) \, d\m_{\ab}(x) - \int F(\infty,x) \, d\m_{\ab}(x) ;
\end{equation*}
here the integrals are finite.
\end{lem}

\begin{proof}
First note that for each $x \in (-1,1)^d$ the desired limits
exist, because we may write
$$
F(t+T,x) - F(t,x)  = \int_{t}^{t+T} \partial_s F(s,x) \, ds
$$
and the conclusion follows by the condition (c).

Now, observe that by \eqref{divergent}
\begin{align*}
& \JE^{\ab} F(t,x) \\& = \partial^2_t F(t,x) + \sum _{j=1}^{d}
    (1-x_j)^{-\alpha_j}(1+x_j)^{-\beta_j} \partial_{x_j} \big[ (1-x_j)^{\alpha_j+1}
        (1+x_j)^{\beta_j+1} \partial_{x_j} F(t,x) \big].
\end{align*}
For $0< \varepsilon < 1$ let $\mathcal{D}_{\varepsilon} = (\varepsilon,-\ln \varepsilon) \times
    (-1+\varepsilon,1-\varepsilon)^d.$ Given $j \in \{1,\ldots, d\}$ we have
\begin{align*}
& \int_{-1+\varepsilon}^{1-\varepsilon} (1-x_j)^{-\alpha_j}(1+x_j)^{-\beta_j} \partial_{x_j}
        \big[ (1-x_j)^{\alpha_j+1} (1+x_j)^{\beta_j+1} \partial_{x_j} F(t,x) \big]
        \, d\m_{(\alpha_j,\beta_j)}(x_j)\\
& = \varepsilon^{\alpha_j+1}(2-\varepsilon)^{\beta_j+1}
    \partial_{x_j} F(t,x)\big|_{x_j=1-\varepsilon} -
    (2-\varepsilon)^{\alpha_j+1} \varepsilon^{\beta_j+1}
    \partial_{x_j} F(t,x)\big|_{x_j=-1+\varepsilon},
\end{align*}
hence, by (b),
\begin{align*}
& \bigg|
\int\!\!\!\int_{\mathcal{D}_{\varepsilon}} t (1-x_j)^{-\alpha_j}(1+x_j)^{-\beta_j} \partial_{x_j}
        \big[ (1-x_j)^{\alpha_j+1} (1+x_j)^{\beta_j+1} \partial_{x_j} F(t,x) \big]
        \, d\m_{\ab}(x) dt \bigg| \\
& \le c\, (\ln \varepsilon)^2 \big( \varepsilon^{\alpha_j+1}
(2-\varepsilon)^{\beta_j+1} +
    (2-\varepsilon)^{\alpha_j+1} \varepsilon^{\beta_j+1}  \big).
\end{align*}
Therefore,
\[
\int\!\!\!\int_{\mathcal{D}_{\varepsilon}} t \sum _{j=1}^{d}
    (1-x_j)^{-\alpha_j}(1+x_j)^{-\beta_j} \partial_{x_j} \big[ (1-x_j)^{\alpha_j+1}
        (1+x_j)^{\beta_j+1} \partial_{x_j} F(t,x) \big] \, d\m_{\ab}(x)dt
\]
tends to $0$ as $\varepsilon \longrightarrow 0^{+}.$
This, together with the monotone or the dominated convergence theorem, implies
\[
\int_0^{\infty} \int t \, \JE^{\ab} F(t,x) \, d\m_{\ab}(x)dt
= \lim _{\varepsilon \to 0^{+}} \int\!\!\!\int_{\mathcal{D}_{\varepsilon}} t
\partial^2_{t} F(t,x) \, dt d\m_{\ab}(x).
\]
On the other hand, integrating by parts we obtain
\[
\int_{\varepsilon}^{-\ln \varepsilon} t \partial^2_t F(t,x) \, dt  =
    F(\varepsilon ,x) - F(-\ln \varepsilon,x)
        + t \partial_t F(t,x) \big|_{t=\varepsilon}^{t=-\ln \varepsilon}.
\]
By (c) the absolute value of the last term is estimated from above by
$\phi(\varepsilon) + \phi(-\ln \varepsilon).$
Since $\lim_{t \to 0^{+}} \phi(t) = \lim_{t \to \infty} \phi(t) = 0,$
the proof is finished with the aid of (a) and the bounded convergence theorem.
\end{proof}

\begin{prop} \label{lem_sat_by_P}
Lemma \ref{lem_parts} may be applied to the function
\[
F(t,x) = \big(\P^{\ab}_t f(x)\big)^p,
\]
where $p\ge 1$ and $f$ is an arbitrary nonnegative function from $C_c^2((-1,1)^d).$
\end{prop}

\begin{proof}
By the subordination principle and Proposition \ref{diff_equations}, we get
\begin{equation*}
\partial_t \P^{\ab}_t f(x) =
\frac{1}{\sqrt{\pi}} \int_0^{\infty} \frac{{\e}^{-u}}{\sqrt{u}} \partial_t
        \Big(T^{\ab}_{{t^2}\slash (4u)} f(x) \Big) \, du
 = \frac{-1}{2\sqrt{\pi}} \int_0^{\infty} \frac{{\e}^{-u}}{\sqrt{u}} \frac{t}{u}
        \J^{\ab} \Big(T^{\ab}_{{t^2} \slash (4u)} f(x) \Big) \, du.
\end{equation*}
Interchanging the order of differentiation and integration above is justified by the dominated
convergence theorem, using (see also the considerations below)
\begin{align*}
\partial_{x_j} T^{\ab}_t f(x) & = {\e}^{-t(\alpha_j+\beta_j+2)} T^{(\alpha+e_j,\beta+e_j)}_t
        \big( \partial_{x_j} f \big)(x),\\ \nonumber
\partial^2_{x_j} T^{\ab}_t f(x) & = {\e}^{-2t(\alpha_j+\beta_j+3)}
        T^{(\alpha+2e_j,\beta+2e_j)}_t \big( \partial^2_{x_j} f \big)(x).
\end{align*}
These identities are easily verified for Jacobi polynomials, and for
$f \in C^2_c((-1,1)^d)$ they are checked by term by term
differentiation of the series defining $T^{\ab}_t f$; 
see the proof of Proposition \ref{diff_equations}. Thus
\begin{align*}
&\partial_t \P^{\ab}_t f(x) \\
& = \frac{1}{2\sqrt{\pi}} \int_0^{\infty} \frac{{\e}^{-u}}{\sqrt{u}} \frac{t}{u}
        {\e}^{-t^2 (\alpha_j+\beta_j+3)\slash (2u)} \sum_{j=1}^d (1-x^2_j)
        T^{(\alpha+2e_j,\beta+2e_j)}_{{t^2} \slash (4u)}
        \big(\partial^2_{x_j} f \big) (x)  \, du\\
& \quad + \frac{1}{2\sqrt{\pi}} \int_0^{\infty} \frac{{\e}^{-u}}{\sqrt{u}} \frac{t}{u}
        {\e}^{-t^2 (\alpha_j+\beta_j+2)\slash (4u)} \sum_{j=1}^d (\beta_j-\alpha_j)
        T^{(\alpha+e_j,\beta+e_j)}_{{t^2} \slash (4u)}
        \big(\partial_{x_j} f \big) (x)  \, du\\
& \quad - \frac{1}{2\sqrt{\pi}} \int_0^{\infty} \frac{{\e}^{-u}}{\sqrt{u}} \frac{t}{u}
        {\e}^{-t^2 (\alpha_j+\beta_j+2)\slash (4u)} \sum_{j=1}^d (\alpha_j+\beta_j+2) x_j
        T^{(\alpha+e_j,\beta+e_j)}_{{t^2} \slash (4u)}
        \big(\partial_{x_j} f \big) (x)  \, du\\
& \equiv  \mathcal{I}_1 + \mathcal{I}_2 + \mathcal{I}_3.
\end{align*}
Since the $T^{\ab}_t$ are contractions on $L^{\infty}((-1,1)^d)$ and $f$ has bounded first
and second order derivatives, we have
$$
|\mathcal{I}_1| + |\mathcal{I}_2| + |\mathcal{I}_3| \le 
	c \, \sum _{j=1}^d \int_0^{\infty} \frac{{\e}^{-u}}{\sqrt{u}}
       \frac{t}{u} {\e}^{-t^2 (\alpha_j+\beta_j+2)\slash (4u)} \, du =
       c \, \sum _{j=1}^d {\e}^{-t\sqrt{\alpha_j+\beta_j+2}}.
$$
Consequently,
\[
\big| t \partial_t \P^{\ab}_t f(x)\big| \le c \, t \, {\e}^{-t \min_j \sqrt{\alpha_j+\beta_j + 2}}.
\]
Since $f$ is bounded and $\P^{\ab}_t$ are contractions on $L^{\infty}((-1,1)^d)$
in view of the same property for $T^{\ab}_t$ and the subordination principle,
the hypotheses (a) and (c) of Lemma \ref{lem_parts} are satisfied
with $\phi(t)=c t{\e}^{-t\min_j \sqrt{\alpha_j+\beta_j + 2}}$.
Concerning (b), we have
\begin{align*}
\big| \partial_{x_j} \P^{\ab}_t f(x) \big| & = \bigg|
\frac{1}{\sqrt{\pi}} \int_0^{\infty} \frac{{\e}^{-u}}{\sqrt{u}}
        {\e}^{-t^2 (\alpha_j+\beta_j+2) \slash (4u)}
        T^{(\alpha+e_j,\beta+e_j)}_{{t^2} \slash (4u)}
        \big(\partial_{x_j} f \big) (x)  \, du \bigg|\\
& \le {c} \int_0^{\infty} \frac{{\e}^{-u}}{\sqrt{u}}
        {\e}^{-t^2(\alpha_j+\beta_j+2)\slash (4u)}
        \, du \\
& =  c\, {\e}^{-t\sqrt{\alpha_j+\beta_j+2}}.
\end{align*}
The inequality $\JE^{\ab} \big[(\P^{\ab}_t f(x))^p \big] \ge 0$
is justified with the aid of Proposition \ref{diff_equations} and Lemma \ref{lem_1}.
Indeed, an application of Lemma \ref{lem_1} is possible since here $\P^{\ab}_t f$ is strictly
positive if $f(x_0)>0$ for some $x_0 \in (-1,1)^d;$
this, in turn, follows by the subordination formula and the strict positivity of the kernel $G^{\ab}_t(x,y)$.
\end{proof}

\begin{prop} \label{prop_g}
Let $f \in C_c^2((-1,1)^d)$ be nonnegative. Then $g(f)$ is bounded on $(-1,1)^d.$
\end{prop}

\begin{proof}
We apply the estimates of $|\partial_t \P^{\ab}_t f(x)|$ and $|\partial_{x_j} \P^{\ab}_t f(x)|$
obtained in the proof of Proposition \ref{lem_sat_by_P}.
\end{proof}

\begin{prop} \label{lem_sat_by_PP}
Lemma \ref{lem_parts} may be applied to the function
\[
F(t,x) = \big(\P^{\ab}_t f(x)\big)^2 \P^{\ab}_t h(x),
\]
with arbitrary nonnegative $f,h \in C_c^2((-1,1)^d).$
\end{prop}

\begin{proof}
Items (a)--(c) are verified as in the proof of Proposition \ref{lem_sat_by_P}.
It remains to prove the integrability condition.

Given $C^2$ functions $F,G \colon (0,\infty) \times (-1,1)^d \mapsto
(-\infty,\infty),$ one has
\begin{equation} \label{mult_form}
\JE^{\ab} (FG) = (\JE^{\ab}F) G + (\JE^{\ab} G) F
        + 2 \langle \nabla_{\ab} F, \nabla_{\ab} G \rangle.
\end{equation}
Therefore,
\begin{align*}
& \big|\JE^{\ab} [(\P^{\ab}_t f(x))^2 \P^{\ab}_t h(x)] \big| \\ & \le
    \JE^{\ab} (\P^{\ab}_t f(x))^2 \, \P^{\ab}_t h(x) +
    2 \big|\nabla_{\ab} (\P^{\ab}_t f(x))^2\big| \,
    \big|\nabla_{\ab} \P^{\ab}_t h(x)\big|,
\end{align*}
since
$\JE^{\ab} (\P^{\ab}_t h(x)) = 0$ by Proposition \ref{diff_equations} and the
first quantity in the right-hand side is nonnegative in view of Lemma \ref{lem_1}.
Now, observe that
\begin{align*}
& \int_0^{\infty} \int t \,
        \JE^{\ab} (\P^{\ab}_t f(x))^2 \, \P^{\ab}_t h(x) \,
        d\m_{\ab}(x) dt \\ & \le
        c \int_0^{\infty}\int t \, \JE^{\ab} (\P^{\ab}_t f(x))^2\,
        d\m_{\ab}(x) dt,
\end{align*}
and the last term is finite by Proposition \ref{lem_sat_by_P}. Further, by the Cauchy-Schwarz
inequality and the fact that $\P^{\ab}_t f(x)$ is bounded, we obtain
\begin{align*}
& \int_0^{\infty} \int t \big|\nabla_{\ab} (\P^{\ab}_t f(x))^2\big| \,
    \big|\nabla_{\ab} \P^{\ab}_t h(x)\big| d\m_{\ab}(x) dt \\ & \le
    c  \int g(f)(x)\, g(h)(x) \, d\m_{\ab}(x),
\end{align*}
and the last integral is finite by Proposition \ref{prop_g}.
\end{proof}

\begin{proof}[Proof of Theorem \ref{main_theorem} (a); the case $1<p\le 2$] \ \\
Apart from minor changes, the proof relies on a classic reasoning, see \cite[p.\,51]{St}.
Observe first that it is sufficient to prove the desired estimate for all nonnegative
$f \in C_c^2((-1,1)^d).$ Then the theorem is justified by standard arguments,
that is decomposition into positive and negative real and imaginary parts, 
approximation of each part by a sequence of nonnegative smooth compactly supported functions,
and an application of Fatou's lemma.

Assume that $f \in C_c^2((-1,1)^d), \;f \ge 0$ and $f(x_0) \neq 0$ for some $x_0 \in (-1,1)^d.$
As pointed out in the proof of Proposition \ref{lem_sat_by_P},
we may apply Lemma \ref{lem_1} to $F(t,x) = \P^{\ab}_t f(x),$ getting
$\JE^{\ab} (\P^{\ab}_t f(x))^p \ge 0$ and 
\begin{align*}
p(p-1)[g(f)(x)]^2 & =  p(p-1) \int _{0}^{\infty} t \big|\nabla_{\ab} (\P^{\ab}_t f(x))\big|^2 \, dt\\
& =
 \int_{0}^{\infty} t(\P^{\ab}_t f(x))^{2-p}\; \JE^{\ab}
        (\P^{\ab}_t f(x))^p \, dt\\
& \le
  \big[\P^{\ab}_{*} f(x) \big]^{2-p} \int_{0}^{\infty} t \,
        \JE^{\ab} (\P^{\ab}_t f(x))^p \, dt.
\end{align*}
Thus, by H\"older's inequality,
(\ref{max_ineq}) and Proposition \ref{lem_sat_by_P} we obtain
\begin{align*}
\|g(f)\|^p_{p,\ab} & \le
c_p \int \big[\P^{\ab}_{*} f(x) \big]
            ^{p(1-p\slash 2)} \bigg( \int_0^{\infty} t \, \JE^{\ab}
            (\P^{\ab}_t f(x))^p \, dt \bigg)^{p\slash 2} d\m_{\ab}(x)\\
& \le
c_{p} \|f\|^{p(1-p\slash 2)}_{p,\ab} \bigg( \int ^{\infty}_0
        \int t \, \JE^{\ab} (\P^{\ab}_t f(x))^p \,
        d\m_{\ab}(x) dt \bigg)^{p\slash 2} \\
& =
c_{p} \|f\|^{p(1-p\slash 2)}_{p,\ab} \bigg(\int f(x)^p \,
        d\m_{\ab}(x) - \int (\P^{\ab}_{\infty} f(x))^p \,
                d\m_{\ab}(x) \bigg)^{p \slash 2}.
\end{align*}
Clearly, the last expression is not greater than $c_{p} \|f\|^{p}_{p,\ab}$.
The conclusion follows.
\end{proof}

\begin{proof}[Proof of Theorem \ref{main_theorem} (a); the case $p>2$] \ \\
We begin with some basic observations.
Given a reasonable, real-valued function $f,$ by the Cauchy-Schwarz inequality  we have
\[
\big(\P^{\ab}_t f(x)\big)^2 \le {\P}^{\ab}_t(f^2)(x)\,
    {\P}^{\ab}_t {\bf{1}}(x) = {\P}^{\ab}_t(f^2)(x).
\]
Combined with Corollary \ref{cor_1}, this gives
\[
\big(\widetilde{\P}^{\ab,i}_t f(x)\big)^2 \le
    \big(\widetilde{\P}^{\ab,i}_t |f|(x)\big)^2
 \le \big(\P^{\ab}_t |f|(x)\big)^2
\le {\P}^{\ab}_t (f^2)(x).
\]
Assume that $f \in C_c^1((-1,1)^d).$ Then for any $t_1,t_2>0$
$$
\partial_t \P^{\ab}_{t_1+t_2} f(x) = \P^{\ab}_{t_1} \big(\partial_t \P^{\ab}_{t_2} f \big) (x)
$$
(here and later on we write $\partial_t \P^{\ab}_{\tau}f$ to denote the derivative in $t$ of
$\P^{\ab}_t f$ taken at the point $\tau$, and this even when $\tau = t\slash 2$).
In particular, letting $t_1=t_2=t\slash 2$ it follows that
$$
\partial_t \P^{\ab}_{t} f(x) = \P^{\ab}_{t\slash 2} \big(\partial_t \P^{\ab}_{t \slash 2}
    f \big) (x).
$$
Moreover,
\begin{equation} \label{ex_ord}
\delta_j (\P^{\ab}_t f)(x) = \widetilde{\P}^{\ab,j}_t (\delta_j f)(x),\qquad
    j=1,\ldots,d,
\end{equation}
which is directly verified for Jacobi polynomials by means of Lemma \ref{lem_intr} and \eqref{jac_diff}.
For $f \in C_b^1((-1,1)^d)$ the last identity follows by
differentiating term by term the series of $\P^{\ab}_t f.$ Therefore
\[
\delta_j(\P^{\ab}_{t+s} f)(x) = \widetilde{\P}^{\ab,j}_t (\delta_j \P^{\ab}_s f)(x),
                \quad \quad j=1,\ldots,d.
\]

Let $f,h$ be nonnegative functions in $C_c^2((-1,1)^d).$
Using the above observations, the symmetry of $\P^{\ab}_t$ and Lemma \ref{lem_1} with $p=2,$ we write
\begin{align*}
& \int  [g(f)(x)]^2 h(x) \, d\m_{\ab}(x) \\
& = \int 
    \int_0^{\infty} t \big|\nabla_{\ab} \P^{\ab}_t f(x)\big|^2 h(x) \, dt d\m_{\ab}(x)\\
& =  \int \int_0^{\infty} \! t\bigg[ \! \Big( \P^{\ab}_{t\slash 2} \big(\partial_t
    \P^{\ab}_{t \slash 2} f \big)\!(x) \Big)^2
    \! + \sum _{j=1}^d \! \Big(
    \widetilde{\P}^{\ab,j}_{t \slash 2} \big( \delta_j \P^{\ab}_{t \slash 2} f \big)(x)
    \Big)^2 \bigg] h(x) \, dt d\m_{\ab}(x)\\
& \le   \int \int_0^{\infty} \! t \bigg[ \P^{\ab}_{t\slash 2} \Big(
    \big[\partial_t \P^{\ab}_{t \slash 2} f \big]^2\Big)(x)
    + \sum _{j=1}^d \P^{\ab}_{t \slash 2} \Big( \big[\delta_j
    \P^{\ab}_{t \slash 2} f \big]^2 \Big) (x)  \bigg] h(x) \, dt d\m_{\ab}(x)\\
& =  \int_0^{\infty} t \int \big| \nabla_{\ab} \P^{\ab}_{t\slash 2}
    f(x) \big|^2 \P^{\ab}_{t \slash 2}h(x) \, d\m_{\ab}(x)dt \\
& =  2  \int_0^{\infty} t \int \JE^{\ab} \big[\P^{\ab}_{t}
    f(x) \big]^2 \P^{\ab}_{t}h(x) \, d\m_{\ab}(x)dt.
\end{align*}
Now, by (\ref{mult_form}) and the identity $\JE^{\ab}(\P^{\ab}_t h(x)) = 0,$
the last double integral equals
\begin{align*}
& \int_0^{\infty} t \int \JE^{\ab} \big[ \big( \P^{\ab}_t
    f(x) \big)^2 \P^{\ab}_t h(x) \big] \, d\m_{\ab}(x)dt\\
    & - 2 \int_0^{\infty} t \int
    \Big\langle \nabla_{\ab} \P^{\ab}_t h(x),
    \nabla_{\ab} (\P^{\ab}_t f(x))^2 \Big\rangle \, d\m_{\ab}(x)dt \\
    &  \equiv  \mathcal{J}_1 - \mathcal{J}_2.
\end{align*}
Applying Lemma \ref{lem_parts} to $\mathcal{J}_1$ (this is legitimate by
Proposition \ref{lem_sat_by_PP}) and the Cauchy-Schwarz inequality
to $\mathcal{J}_2$, we arrive at 
\begin{align*}
& \int [g(f)(x)]^2 h(x) \, d\m_{\ab}(x) \\ & \le
    2  \int f^2(x)h(x) \, d\m_{\ab}(x) + 8
    \int \P^{\ab}_{*}f(x) \, g(f)(x) \, g(h)(x) \, d\m_{\ab}(x).
\end{align*}
Assume that $p\ge 4,\; (2\slash p) + (1\slash q) =1$ and $\|h\|_{q,\ab} \le 1.$
By H\"older's inequality for three functions, (\ref{max_ineq}) and
Theorem \ref{main_theorem} (a) with $p$ replaced by $q\le 2$, we obtain
\[
\langle g(f)^2 , h \rangle_{{\ab}} \le
    c_{p} \,  \big( \|f\|^2_{p,\ab}
    + \|g(f)\|_{p,\ab} \|f\|_{p,\ab} \big).
\]
Consequently, also $\|g(f)\|^2_{p,\ab}$ is dominated by the right hand side above.
This implies the desired estimate for $p \ge 4.$ For $2 < p < 4$ the result follows by
Marcinkiewicz' interpolation theorem.
\end{proof}


\section{Riesz transforms and conjugate Poisson integrals}
\label{sec:Riesz}

Recall that the Riesz-Jacobi transform
$\mathcal{R}^{\ab} = \big(R_1^{\ab},\ldots,R_d^{\ab}\big)$
is formally given by \eqref{def_riesz}, which
makes sense for the dense subset of $L^p(d\m_{\ab}), \; 1 \le p < \infty,$
consisting of all polynomials, see \eqref{riesz_pol}.
Assuming to begin with that $f$ is a polynomial,
we define its \emph{conjugate Poisson integrals}  $U^{\ab,i}_t f, \; i=1,\ldots,d,\; t>0,$ by
\begin{equation*}
U^{\ab,i}_t f = \widetilde{\P}^{\ab,i}_t R^{\ab}_{i} f.
\end{equation*}
This definition
is analogous to those for Hermite and Laguerre expansions
and is well motivated by the following set of Cauchy--Riemann type equations:
\begin{align} \label{cr1_jac_pol}
\delta_j U^{\ab,i}_t f & = \delta_i U^{\ab,j}_t f,
\quad \quad i,j = 1, \ldots ,d, \\ \label{cr2_jac_pol}
\delta_j \P^{\ab}_t f & = -\partial_t U^{\ab,j}_t f,
\quad \quad j=1, \ldots ,d, \\ \label{cr3_jac_pol}
\sum^d_{j=1}\delta^{*}_j U^{\ab,j}_t f & = - \partial_t \P^{\ab}_t f.
\end{align}
Moreover, we have the harmonicity relations
\begin{equation} \label{cr5_jac_pol}
\big(\partial^2_t - M^{\ab}_{j}\big) U^{\ab,j}_t f =  0, \quad \quad j=1,\ldots,d.
\end{equation}
The verification of (\ref{cr1_jac_pol})--(\ref{cr5_jac_pol}) is straightforward
when $f$ is a (Jacobi) polynomial.
Indeed, for $i=1,\ldots,d,$ we have (see Section \ref{sec:supplementary})
\begin{align*}
\P^{\ab}_t \jac^{\ab}_k & = {\e}^{-t\lambda_k^{1\slash 2}} \jac^{\ab}_k ,\\
M^{\ab}_{i} \big(\sqr_i\jac^{(\alpha+e_i,\beta+e_i)}_{k-e_i}\big) & =
    \lambda_k \sqr_i\jac^{(\alpha+e_i,\beta+e_i)}_{k-e_i},\\
\widetilde{\P}^{\ab,i}_t \big(\sqr_i\jac^{(\alpha+e_i,\beta+e_i)}_{k-e_i}\big) & =
    {\e}^{-t\lambda_k^{1\slash 2}} \sqr_i\jac^{(\alpha+e_i,\beta+e_i)}_{k-e_i},\\
U^{\ab,i}_t \jac^{\ab}_k & =
    \frac{1}{2}\lambda_k^{-1\slash 2}(k_i+\alpha_i+\beta_i+1)
    {\e}^{-t\lambda_k^{1\slash 2}} \sqr_i\jac^{(\alpha+e_i,\beta+e_i)}_{k-e_i},
\end{align*}
the last equation being valid for $k_i>0,$ and if $k_i=0$ we have $U^{\ab,i}_t \jac^{\ab}_k =0.$
Now the identities (\ref{cr1_jac_pol}) and (\ref{cr2_jac_pol}) are easily verified by means of
(\ref{diff_jac_pol}). To get (\ref{cr3_jac_pol}), observe that
\begin{equation} \label{adj_diff}
\delta^{*}_i \big(\sqr_i\jac^{(\alpha+e_i,\beta+e_i)}_{k-e_i}\big) = 2 k_i \jac^{\ab}_{k},
\end{equation}
which follows by
(\ref{diff_jac_pol}) and the fact that Jacobi polynomials $\jac^{\ab}_k$ are eigenfunctions
of $\J^{\ab}.$ Checking (\ref{cr5_jac_pol}) causes no difficulties.
Finally, it is interesting to observe that
\begin{equation} \label{spf2}
U_t^{\ab,j}f = R^{\ab}_j \P^{\ab}_t f, \qquad j=1,\ldots,d,
\end{equation}
for all polynomials $f$, which again is immediately verified for Jacobi polynomials.

The standard Cauchy-Riemann equations in the complex plane tell us when $u$ and $v$ are the real and
imaginary parts of an analytic function. It is well known that they are also equivalent to the
property that $(u,v)$ is the gradient of a harmonic function, at least in a simply connected region.
It is remarkable that a similar property holds also in our setting, so that
$$
\big( -\P^{\ab}_t \Pi_0 f(x), U_t^{\ab,1}f(x),\ldots, U^{\ab,d}_t f(x) \big)
$$
is the gradient in the sense of $\nabla_{\!\ab}$ of a function $F(t,x)$ that is harmonic
with respect to $\mathbb{J}^{\ab} = \partial_t^2 - \J^{\ab}.$ To see this, let
\begin{equation*}
F(t,x) = (\J^{\ab})^{-1\slash 2} \P^{\ab}_t \Pi_0 f(x) =
    \sum_{|k|>0} a_k(f) \lambda_k^{-1\slash 2} \e^{-t \lambda_k^{1\slash 2}}
 P^{\ab}_k(x)
\end{equation*}
and differentiate term by term to obtain the relevant identity.

We now state the main result of the paper.

\begin{thm}  \label{main_result_of_paper}
Assume that $1 < p < \infty$ and $\alpha,\beta \in [-1\slash 2, \infty)^d.$
There exists a constant $c_p$ (depending neither on the dimension $d$
nor on the type multi-indices $\alpha,\beta$) such that
\[
\big{\|} {R}_i^{\ab} f \big{\|}_{p,\ab}  \le
    c_p \|f\|_{p,\ab},
\]
for all $i=1,\ldots,d$, and all polynomials $f$ in $(-1,1)^d.$ Consequently, the operators
$R^{\ab}_i$ and $U^{\ab,i}_t, \; i=1,\ldots,d,$ $t>0,$ initially defined on the dense
subset of $L^p(d\m_{\ab})$ consisting of all polynomials, 
extend uniquely to bounded linear operators in $L^p(d\m_{\ab})$.
\end{thm}

\begin{proof}
Given $i=1,\ldots,d$, by (\ref{cr2_jac_pol}) we have
$\partial_t \widetilde{\P}_t^{\ab,i} (R^{\ab}_i f) = - \delta_i \P^{\ab}_t f$, hence
\[
\widetilde{g}_i({R}^{\ab}_i f)(x) \le g(f)(x).
\]
Ergo the conclusion follows by Theorem \ref{main_theorem}.
\end{proof}

In what follows we use the symbols $R^{\ab}_i$, $U^{\ab,i}_t$
to denote also the extensions in Theorem \ref{main_result_of_paper}.
Moreover, we may assume that the extensions $U^{\ab,i}_t$ are given by
$\widetilde{\P}^{\ab,i}_t R^{\ab}_{i}$, so that they map $L^p(d\m_{\ab})$
into smooth functions, see Proposition \ref{diff_equations}.

\begin{cor} \label{conjugate_result}
Let $1 < p < \infty$, $\alpha,\beta \in [-1\slash 2, \infty)^d$ and $i \in \{1,\ldots,d\}$. Then
\begin{enumerate}
\item[(a)]
    there exists a constant $c_p$ such that for all $f \in L^p(d\m_{\ab})$
    \[
    \big\|\sup_{t>0} |U^{\ab,i}_t f| \big\|_{p,\ab} 
    \le c_p \|f\|_{p,\ab}, \quad \quad t>0;
    \]    
\item[(b)]
    for all $f \in L^p(d\m_{\ab})$
    \[
    U^{\ab,i}_t f \to R^{\ab}_i f, \quad t \to 0^{+},
    \]
    the convergence being both in $L^p(d\m_{\ab})$ and almost everywhere;
\item[(c)]
    the family $\{U^{\ab,i}_t \}_{t \ge 0}$,
    with $U^{\ab,i}_0 = R_i^{\ab}$,
    is strongly continuous in $L^p(d\m_{\ab})$.
\end{enumerate}
\end{cor}

\begin{proof}
Item (a) is a consequence of Theorem \ref{main_result_of_paper}, (\ref{max_ineq})
and the fact that
\begin{equation} \label{X}
\sup_{t>0} |U^{\ab,i}_t f(x)| \le \P^{\ab}_{*} \big(R^{\ab}_i f\big)(x),
    \quad \quad x \in (-1,1)^d.
\end{equation}
Statements (b) and (c) are justified by standard arguments with the aid of (a) and \eqref{X}.
\end{proof}

\begin{rem} \label{rem_1}
When $p=2$ it is easy to compute that for the full range of $\alpha,\beta \in (-1,\infty)^d$
\[
\big{\|} |\mathcal{R}^{\ab} f |_{\ell^2} \big{\|}_{2,\ab}  =
     \big{\|}\Pi_0 f\big{\|}_{2,\ab}, \quad \quad f \in L^2(d\m_{\ab}),
\]
with $|\cdot|_{\ell^2}$ denoting the Euclidean norm in $\R^d$.
Indeed, since for $f \in L^2(d\m_{\ab})$ 
\[
R^{\ab}_i f = \sum _{|k|>0} a_k(f)
     \frac{1}{2}(\lambda_k)^{-1\slash 2}(k_i+\alpha_i+\beta_i+1)
     \sqr_i \jac_{k-e_i}^{(\alpha+e_i,\beta+e_i)},
\]
by Parseval's identity and \eqref{l2norm} we get
\begin{align*}
\big{\|} |\mathcal{R}^{\ab} f |_{\ell^2} \big{\|}^2_{2,\ab}
& =  \sum_{i=1}^d \sum_{|k|>0} |a_k(f)|^2 \frac{(k_i+\alpha_i+\beta_i+1)^2}{4\lambda_k}
    \big\|\sqr_i\jac^{(\alpha+e_i,\beta+e_i)}_{k-e_i}\big\|^2_{2,\ab}\\
&   =
 \sum_{|k|>0} |a_k(f)|^2 \sum_{i=1}^d \frac{k_i(k_i+\alpha_i+\beta_i+1)}{\lambda_k}
    \|\jac^{\ab}_k\|^2_{2,\ab}\\
&    =
\sum_{|k|>0} |a_k(f)|^2 \|\jac^{\ab}_k\|^2_{2,\ab} \\
&    =  \|\Pi_0 f\|^2_{2,\ab}.
\end{align*}
A similar computation shows that for all $\alpha,\beta \in (-1,\infty)^d$,
$$
\Big{\|} \big|\big(U^{\ab,1}_t f,\ldots,U^{\ab,d}_t f\big)
    \big|_{\ell^2} \Big{\|}_{2,\ab} \le \|\Pi_0 f\|_{2,\ab}, \qquad t>0.
$$
\end{rem}

\begin{rem}
If $1<p<\infty$ and $\alpha,\beta \in [-1\slash 2,\infty)^d$ then the (dimension dependent) estimate
$$
c_{d,p} \|f\|_{p,\ab} \le \big\||\mathcal{R}^{\ab}f|_{\ell^2}\big\|_{p,\ab}, 
\qquad f \in L^p(d\m_{\ab}),
$$
holds under the restriction
$\int f d\m_{\ab} =0$. This follows by a standard duality argument, taking into account
the upper bound from Theorem \ref{main_result_of_paper} and the isometry of Remark \ref{rem_1}.
\end{rem}

We finally verify that the conjugate Poisson operators $U_t^{\ab,i}$ are given on 
$L^p(d\m_{\ab})$ by the appropriate series expansions.

\begin{prop} \label{cpintrep}
Let $\alpha,\beta \in (-1,\infty)^d$ and $f \in L^1(d\m_{\ab})$. For
$i \in \{1,\ldots,d\}$, $t>0$ and $x \in (-1,1)^d$ define
\begin{equation} \label{cpiser}
f^{\ab}_i(t,x) = \sum_{|k|>0} a_k(f) \frac{1}{2}\lambda_k^{-1\slash 2}(k_i+\alpha_i+\beta_i+1)
    e^{-t \lambda_k^{1\slash 2}} \sqr_i(x) \jac_{k-e_i}^{(\alpha+e_i,\beta+e_i)}(x).
\end{equation}
The above series converges pointwise, $f^{\ab}_i(t,x)$ is a $C^{\infty}$ function
of $(t,x) \in (0,\infty) \times (-1,1)^d$ and \eqref{cr1_jac_pol}-\eqref{cr5_jac_pol} hold with
$U_t^{\ab,i}f(x)$ replaced by $f_i^{\ab}(t,x)$.
Moreover, when $\alpha,\beta \in [-1\slash 2, \infty)^d$ and $f \in L^p(d\m_{\ab})$ for some
$p>1$, then 
\begin{equation} \label{cpieq}
U^{\ab,i}_t f(x) = f^{\ab}_i(t,x), \qquad t>0, \quad x \in (-1,1)^d.
\end{equation}
\end{prop}

\begin{proof}
Pointwise convergence of the series in \eqref{cpiser} is justified like
the convergence of the series defining $T^{\ab}_t$ in \eqref{def_heat_ser}.
The remaining statements, except the last one, are proved by arguments similar
to those of the proof of Proposition \ref{diff_equations}.
Thus it remains to show \eqref{cpieq}.

Let $f \in L^p(d\m_{\ab})$ and take a sequence $\{f_n\}$ of polynomials
converging in $L^p(d\m_{\ab})$ to $f$. 
Then $U^{\ab,i}_t f_n \to U^{\ab,i}_t f$ in $L^p(d\m_{\ab})$ by Corollary
\ref{conjugate_result} (a). On the other hand, since $U^{\ab,i}_t f_n$
coincides with $(f_n)_i^{\ab}(t,\cdot)$, we see that
$$
f^{\ab}_i(t,x) - U^{\ab,i}_t f_n(x) =
\sum_{|k|>0} a_k(f-f_n) \frac{k_i+\alpha_i+\beta_i+1}{2\lambda_k^{1\slash 2}}
    e^{-t \lambda_k^{1\slash 2}} \sqr_i(x) \jac_{k-e_i}^{(\alpha+e_i,\beta+e_i)}(x).
$$
Now using \eqref{jac_est}, H\"older's inequality applied to $a_k(f-f_n)$ and
\eqref{L2normest}, we can estimate
the $L^p(d\m_{\ab})$ norm of the last sum by a constant times the
$L^p(d\m_{\ab})$ norm of $f-f_n$. This implies that $U^{\ab,i}_t f_n$ converges
to $f^{\ab}_i(t,\cdot)$ in $L^p(d\m_{\ab})$. Therefore the functions
$U^{\ab,i}_t f$ and $f^{\ab}_i(t,\cdot)$ are equal in $L^p(d\m_{\ab})$ and
hence almost everywhere on $(-1,1)^d$. The desired result follows by a continuity argument.
\end{proof}

Note that the norm estimate obtained in the above proof shows also that the mapping
$f \mapsto f^{\ab}_i(t,\cdot)$ is continuous on $L^p(d\m_{\ab})$, $1\le p \le \infty$, for
each fixed $t>0$ and $\alpha,\beta$ in the full range $(-1,\infty)^d$.

We shall now augment the conjugacy scheme for Jacobi expansions by introducing additional Riesz
transforms and conjugate Poisson integrals.
In particular, this will clarify the role played by the $\delta_i^*$, the adjoint Jacobi derivatives.

Recall that formally $R^{\ab}_i = \delta_i (\mathcal{J}^{\ab})^{-1\slash 2}\Pi_0.$
It is natural to think of a supplementary system of Riesz-Jacobi transforms defined by means of $\delta^{*}_i,$
instead of $\delta_i.$ However, a simple replacement of $\delta_i$ by its adjoint in the definition of
$R^{\ab}_i$ would not be appropriate since, in particular, such operators would behave badly
even in $L^2(d\m_{\ab})$; for instance it is easy to check that if $\alpha_i=\beta_i=0$ for some $i$
then the operator $\delta^*_i (\mathcal{J}^{\ab})^{-1\slash 2}\Pi_0$ maps $P^{\ab}_{e_i}$ to a function
which is not in $L^2(d\m_{\ab})$.
It turns out that the operators that fit into our setting are given formally by
$$
\overline{R}^{\ab}_i = \delta^{*}_i ({M}^{\ab}_i)^{-1\slash 2}, \qquad i=1,\ldots,d.
$$
We shall see that $\overline{R}^{\ab}_i$ coincides with the adjoint operator $\big(R^{\ab}_i\big)^{*}$.
Indeed, by the identity \eqref{adj_diff} and Lemma \ref{lem_intr} it follows that
\begin{equation} \label{adj_riesz}
\overline{R}^{\ab}_i \big( \sqr_i P_{k-e_i}^{(\alpha+e_i,\beta+e_i)}\big)
= 2 k_i \lambda_k^{-1\slash 2} P^{\ab}_k.
\end{equation}
Consequently, for $f \in L^2(d\varrho_{\ab})$ with the expansion
$f = \sum_{k\in \N^d} a^i_k(f) \sqr_i P_k^{(\alpha+e_i,\beta+e_i)},$ we have
$$
\overline{R}_i^{\ab} f = 2 \sum_{k \in \N^d} a^i_k(f) (k_i+1) \lambda_{k+e_i}^{-1\slash 2} P^{\ab}_{k+e_i},
$$
the series being convergent in $L^2(d\varrho_{\ab}).$
On the other hand, since formally
$\big(R^{\ab}_i\big)^{*} = \big(\delta_i (\mathcal{J}^{\ab})^{-1\slash 2} \Pi_0\big)^{*}
= \Pi_0 (\mathcal{J}^{\ab})^{-1\slash 2} \delta^{*}_i,$ another use of \eqref{adj_diff}
produces
$$
\big(R^{\ab}_i\big)^{*} \big( \sqr_i P_{k-e_i}^{(\alpha+e_i,\beta+e_i)}\big)
= 2 k_i \lambda_k^{-1\slash 2} P^{\ab}_k.
$$
Thus, in view of Lemma \ref{lem_intr} and Remark \ref{rem_1}, we see that
$\big(R^{\ab}_i\big)^{*} = \overline{R}^{\ab}_i$ in $L^2(d\varrho_{\ab})$. Furthermore,
when $\alpha$ and $\beta$ are such that $R^{\ab}_i$ is bounded on 
$L^{q}(d\varrho_{\ab})$, $1\slash p + 1\slash q=1$, for some $1<p<\infty$, then
the operator $\overline{R}^{\ab}_i$, defined initially on the subspace of polynomials
multiplied by $\sqr_i$, has a bounded extension to $L^p(d\m_{\ab})$ given by the adjoint
$\big(R^{\ab}_i\big)^{*}$ taken in the Banach space sense.
We denote this extension by the same symbol $\overline{R}^{\ab}_i$.

Another straightforward computation with the aid of \eqref{riesz_pol} and \eqref{adj_riesz} furnishes
\begin{equation} \label{cl_an}
\sum_{j=1}^d \overline{R}^{\ab}_j R^{\ab}_j = \Pi_0,
\end{equation}
on the space of all polynomials; thus
this identity is also valid on $L^p(d\varrho_{\ab})$, $1<p<\infty$, under the assumption that 
$\alpha,\beta \in [-1\slash 2,\infty)^d$
(here $\Pi_0 f = f - \int f \,d\varrho_{\ab} \slash \int d\varrho_{\ab}$).
Note that the above formula is an analogue of the well-known relation in the Euclidean setting, where
the classic Riesz transforms $R_j = \partial_j (-\Delta)^{-1\slash 2}$ satisfy $\sum_j R_j^2 = - I.$
In the Jacobi setting, however, the associated partial derivatives are not formally skew-symmetric and
do not commute with the Jacobi operator, and hence we have to take $\overline{R}^{\ab}_j R^{\ab}_j$ rather than $\big(R^{\ab}_j\big)^{2}.$

Passing to conjugate Poisson integrals, recall that
$U^{\ab,i}_t = \widetilde{\P}^{\ab,i}_t R^{\ab}_i.$
We define a supplementary system of conjugate Poisson integrals:
$$
\overline{U}^{\ab,i}_{\! t} = {\P}^{\ab}_t \overline{R}^{\ab}_i, \qquad t>0, \quad i=1,\ldots,d.
$$
Note that this definition makes sense in $L^2(d\varrho_{\ab})$ since, by \eqref{adj_riesz},
$$
\overline{U}^{\ab,i}_t \big( \sqr_i P_{k-e_i}^{(\alpha+e_i,\beta+e_i)} \big)
    = 2 k_i \lambda_k^{-1\slash 2} e^{-t \lambda_k^{1\slash 2}} P^{\ab}_k.
$$

It is easily verified that $\big({U}^{\ab,i}_t\big)^{*} = \overline{U}^{\ab,i}_{\! t}$
in $L^2(d\varrho_{\ab}),$
and the same is true in $L^p(d\varrho_{\ab})$ whenever ${U}^{\ab,i}_t$ is bounded on
$L^{q}(d\varrho_{\ab})$, $1\slash p + 1\slash q =1$.
More precisely, in the $L^p$ case, $\big({U}^{\ab,i}_t\big)^{*}$ is a bounded extension
of
$\overline{U}^{\ab,i}_{\! t}$ defined initially on the subspace of polynomials
multiplied by the factor $\sqr_i$; we denote this extension by the same
symbol $\overline{U}^{\ab,i}_{\! t}$.
Further, for reasonable $f,$ the Cauchy-Riemann type equations
\begin{equation} \label{hh1}
\delta_j^{*} \widetilde{\P}^{\ab,j}_t f = - {\partial_t} \overline{U}^{\ab,j}_{\! t} f,
    \qquad j=1,\ldots,d,
\end{equation}
hold, and in one dimension we also have
\begin{equation} \label{hh2}
\delta_1 \overline{U}^{\ab,1}_{\! t} f = -\partial_t \widetilde{\P}^{\ab,1}_t f.
\end{equation}
Moreover, the harmonicity equations
\begin{equation} \label{hh3}
(\partial_t^2 - \J^{\ab}) \overline{U}^{\ab,j}_{\! t} f = 0,
    \qquad j=1,\ldots,d,
\end{equation}
are satisfied and the relation
\begin{equation} \label{hh4}
\sum_{j=1}^d \overline{U}^{\ab,j}_{\! t} {U}^{\ab,j}_t f = \P^{\ab}_{2t} \Pi_0 f.
\end{equation}
holds. Proving \eqref{hh1}-\eqref{hh3} is immediate
when $f$ is a linear combination of $P_k^{(\alpha+e_j,\beta+e_j)},$ $k \in \N^d,$ multiplied by
$\sqr_j;$ the last identity is easily verified for (Jacobi) polynomials.

Assume that $\alpha$ and $\beta$ are such that the $R^{\ab}_i$ are bounded on
$L^p(d\varrho_{\ab})$ for all $1<p<\infty.$ Then
$\overline{U}^{\ab,i}_{\! t} = {\P}^{\ab}_t (R^{\ab}_i)^{*}$ and, in view of the mapping properties
of the maximal operator ${\P}^{\ab}_{*},$ results analogous to the statements of
Corollary \ref{conjugate_result} follow for $\overline{U}^{\ab,i}_{\! t}.$
Moreover, a result analogous to Proposition \ref{cpintrep} is valid for
$\overline{U}^{\ab,i}_{\! t}$, the proof being similar to that for
$U^{\ab,i}_t$. In particular, it follows that \eqref{hh1}-\eqref{hh4} hold for
all $f \in L^p(d\m_{\ab})$, $1<p<\infty$, provided that
$\alpha,\beta \in [-1\slash 2, \infty)^d$.

\begin{rem}
Many identities of this section may be easily verified in a symbolic way,
which is rigorous at least on a dense subspace of $L^2(d\m_{\ab})$.
For example, we have
$$
\sum_j \overline{R}_j^{\ab} R_j^{\ab} =
\sum_j \big( R_j^{\ab} \big)^* R_j^{\ab} = \sum_j \Pi_0 \big( \J^{\ab} \big)^{-1\slash 2} \delta_j^*
    \delta_j \big( \J^{\ab} \big)^{-1\slash 2} \Pi_0 = \Pi_0
$$
and, with the aid of \eqref{spf2} and the above,
\begin{align*}
\sum_j \overline{U}_t^{\ab,j} U_t^{\ab,j} & =
    \sum_j \big( R_j^{\ab} \P_t^{\ab}\big)^{*} \big(R^{\ab}_j \P^{\ab}_t \big) \\
& = \P_t^{\ab} \Big(\sum_j \big(R^{\ab}_j \big)^{*} R^{\ab}_j \Big) \P^{\ab}_t \\
& = \P^{\ab}_{2t} \Pi_0.
\end{align*}
\end{rem}

\begin{rem}
The scheme of conjugacy introduced in this section, including the supplementary Riesz transforms and
conjugate Poisson integrals, is universal in the sense that it
also fits precisely into the settings of Hermite and Laguerre polynomial expansions.
\end{rem}

\begin{rem}
In the one-dimensional setting, the $L^p$ boundedness of the Riesz-Jacobi transform $R^{\ab}_1$ may
be obtained in
a much shorter way $\!^{\dag}$, by means of Muckenhoupt's general transplantation theorem for Jacobi
series \cite[Theorem 1.14, Corollary 17.11]{Mu3},
see also \cite[Sections 3,4]{Stem} where the idea has recently been applied to study
Riesz transforms for expansions based on orthonormalized Jacobi polynomials $\!^{\ddag}$.
In our setting, this approach delivers also weighted results and, in addition,
the full range $(-1,\infty)$ of $\alpha,\beta$ is allowed.
Consequently, if $d=1$ the results of this section may be improved by assuming $\alpha,\beta >-1$
and considering appropriate weights.
\end{rem}

\footnotetext{$^{\dag}$ We thank Krzysztof Stempak for this remark.\\
          \indent $^{\ddag}$ The apparently more appropriate term
          ``Jacobi functions'' would be, alas, confusing.}


\end{document}